\documentclass[12pt]{article}

\def\@eqnnum{\rm ([section].\theequation)}
\usepackage{amsthm}
\usepackage{amsfonts}
\usepackage{pictexwd,dcpic}
\usepackage{amsmath}
\usepackage{amssymb}
\usepackage{color}
\usepackage{amsbsy}
\numberwithin{equation}{section}

\newtheorem{df}{Definition}[section]
\newtheorem{p}[df]{Proposition}
\newtheorem{corol}[df]{Corollary}
\theoremstyle{definition}
\newtheorem{ex}[df]{Example}

\title{\bf  D-polynomials and Taylor formula\\ in quantum calculus}
\author{Piotr Multarzy\'{n}ski\\{\it Faculty of Mathematics and Information Science}\\
 {Warsaw University of Technology} \\
{\it 00-661 Warsaw, Pl. Politechniki 1, Poland}\\
e-mail: multarz@mini.pw.edu.pl}
\date{}
\begin{document}
\maketitle
\begin{abstract}
Quantum calculus based on the right invertible divided difference
operator $D_{\sigma}^{\tau}$ is proposed here in context of
algebraic analysis \cite{DPR}. The linear operator
$D_{\sigma}^{\tau}$, specified with the help of two fixed maps
$\sigma\; , \tau\colon M\rightarrow M$, generalizes the quantum
derivative operator used in $h$- or $q$-calculus \cite{kac}. In the
domain of $D_{\sigma}^{\tau}$ there are special elements defined as
$D_{\sigma}^{\tau}$-polynomials and the corresponding Taylor formula
is proved.
\end{abstract}
{\bf Keywords:} quantum calculus;  Taylor formula; difference operator; right invertible operator \\
{\bf MSC 2000: 12H10, 39A12, 39A70}

\section{Introduction}
The usual $h$- or $q$-calculus as well as many other types of
quantum calculi are specified by a fixed right invertible linear
operator $D$, the so-called quantum derivative. The right inverses
of $D$ allow us to define the concept of indefinite $D$-integrals.
Then, for example by applying Jackson formula, one can define the
corresponding definite integrals \cite{Jackson, kac}. On the other
hand, the calculus of right invertible operators turns out to be a
part of algebraic analysis developed by D. Przeworska-Rolewicz.  For
a right invertible operator $D$, its right inverses give rise to the
corresponding indefinite $D$-integrals, while definite $D$-integrals
are defined with the help of the corresponding initial operators.
For the comprehensive study of the topic we recommend Ref.
\cite{DPR}.

The general concept of $D$-polynomials defined for a right
invertible linear operator $D$ is analyzed in Section
\ref{Polynomials in alg}. In particular, an interesting result is
that the dimension of the linear space $P_n(D)$ of all
$D$-polynomials of degree less or equal $n$, $n\in\mathbb N$,
additionally depends on the dimension of $ker D$ (the so-called
space of $D$-constants), i.e. $dim P_n(D)=(n+1)\cdot ker D$, which
is infinite if $ker D$ is of infinite dimension. Then, in Section
\ref{Divided} we define the so-called $(\sigma ,\tau )$-quantum
derivative as a divided difference operator $D_{\sigma}^{\tau}$
based on three fixed mappings $\sigma ,\tau\colon M\rightarrow M$
(shifts) and $\theta\colon M\times M\rightarrow\mathbb R$ (tension
function), specifying the essence of quantum calculus considered.
One can show that $D_{\sigma}^{\tau}$ is right invertible
\cite{MULT1}. In analogy to the usual concept of polynomials, their
$(\sigma ,\tau )$-quantum counterparts are defined in Section
\ref{D-polynomials} for the assumed $(\sigma ,\tau )$-quantum
derivative $D_{\sigma}^{\tau}$. Finally, in Section \ref{Taylor} the
corresponding $(\sigma ,\tau )$-quantum Taylor formula is proved
(for analogy with q-calculus see \cite{kac}).
\section{Polynomials in algebraic analysis}\label{Polynomials in alg}
Let $X$ be a linear space over a field $\mathbb K$ and ${\cal L}(X)$
be the family of all linear mappings $D:U\rightarrow V$, for any
$U$, $V$ - linear subspaces of $X$. We shall use the notation
$dom(D)=U$, $codom(D)=V$ and $im\, D=\{Du: u\in U\}$ for the domain,
codomain and image of $D$, correspondingly.
 For any operators $D_1,D_2\in{\cal L}(X)$ and scalars
$k_1,k_2\in\mathbb K$, the linear combination $k_1D_1+k_2D_2$ as
well as the composition $D_1D_2$ are the elements of ${\cal L}(X)$
defined on the corresponding domains
\begin{equation}\label{domlinkomb}
dom(k_1D_1+k_2D_2)=dom(D_1)\cap dom(D_2)\, ,
\end{equation}
and
\begin{equation}\label{domcomposed}
dom(D_1D_2)=D_2^{-1}(dom(D_1))\, .
\end{equation}
The domains of all linear mappings considered in the sequel will be
understood in the sense of formulae (\ref{domlinkomb}),
(\ref{domcomposed}).

Throughout this paper we use the notation
\begin{equation}
{\mathbb N}=\{1,2,3,\ldots\}\;\;\; \mbox{and}\;\;\; {\mathbb
N}_0=\{0,1,2,3,\ldots\}\; .
\end{equation}

Whenever $D_1=\ldots =D_m=D\in {\cal L}(X)$, we shall write
$D^m=D_1\ldots D_m$, for $m\in\mathbb N$,  and additionally
$D^0=I\equiv id_{dom(D)}$.

For any $D\in{\cal L}(X)$ and  $m\in\mathbb N$, we assume the
notation
\begin{equation}
Z_0(D)=\{0\}\;\;\;\mbox{and}\;\;\;Z_m(D)=ker D^m\setminus ker
D^{m-1}\; .
\end{equation}
 Evidently, for any $D\in{\cal L}(X)$
there is
\begin{equation}\label{Zdisjoint}
Z_i(D)\cap Z_j(D)=\emptyset ,
\end{equation}
whenever $i\neq j$, and
\begin{equation}\label{sumZm}
\bigcup_{k=0}^{m}Z_k(D)= ker\, D^m .
\end{equation}
In the sequel we shall use the notation
\begin{equation}\label{ZD}
 Z(D)=ker\,D\, ,
\end{equation}
 and refer to $Z(D)$ as the space of constants for
$D\in{\cal L}(X)$.
\begin{p}\label{linindep}
Let $D\in{\cal L}(X)$, $m\in\mathbb N$, and $Z_i(D)\neq\emptyset$
for $i=1,\ldots ,m$. Then, any elements $u_i\in Z_i(D)$, $i=1,\ldots
, m$, are linearly independent.
\end{p}
\begin{proof}
Consider a linear combination $u=\lambda_1 u_1+\ldots\lambda_m u_m$
and suppose that $u=0$ for some coefficients $\lambda_1,\ldots
,\lambda_m\in{\mathbb K}$. Hence we obtain the sequence of
equations: $D^{k}u=\lambda_{k+1}D^{k}u_{k+1}+\ldots +\lambda_m
D^{k}u_{m}=0$, for $k=1,\ldots ,m-1$. Step by step, from these
equations we compute $\lambda_m=0,\ldots ,\lambda_1 =0$.
\end{proof}

Let us define
\begin{equation}\label{calR(X)}
{\cal R}(X)=\{D\in{\cal L}(X): codom(D)=im\,D\},
\end{equation}
 i.e. each element
$D\in{\cal R}(X)$ is considered to be a surjective mapping (onto its
codomain). Thus, ${\cal R}(X)$ consists of all right invertible
elements.
\begin{df}\label{Rinverse}
 An operator $R\in{\cal L}(X)$ is said to be a right inverse of
 $D\in{\cal R}(X)$ if $dom(R)=im(D)$ and $DR=I\equiv id_{im(D)}$.
By ${\cal R}_{D}$ we denote the family of all right inverses of $D$.
\end{df}
In fact, ${\cal R}_{D}$ is a nonempty family, since for each $y\in
im(D)$ we can select an element $x\in D^{-1}(\{y\})$ and define
$R\in{\cal R}_{D}$ such that $R: y\mapsto x$.

The fundamental role in the calculus of right invertible operators
play the so-called initial operators, projecting the domains of
linear operators onto the corresponding space of their constants.
\begin{df}\label{initial}
Any operator $F\in{\cal L}(X)$, such that $dom(F)=dom(D)$,
$im(F)=Z(D)$ and $F^2=F$ is said to be an initial operator induced
by $D\in{\cal R}(X)$. We say that an initial operator $F$
corresponds to a right inverse $R\in{\cal R}_{D}$ whenever $FR=0$ or
equivalently if
\begin{equation}\label{FbyR}
F=I-RD\, .
\end{equation}
The family of all initial operators induced by $D$ will be denoted
by ${\cal F}_D$.
\end{df}
The families ${\cal R}_{D}$ and ${\cal F}_D$ uniquely determine each
other. Indeed, formula (\ref{FbyR}) characterizes initial operators
by means of right inverses, whereas formula
\begin{equation}\label{RbyF}
R=R'-FR'\, ,
\end{equation}
which is independent of $R'$, characterizes right inverses by means
of initial operators. Both families ${\cal R}_D$ and ${\cal F}_D$
are fully characterized by formulae
\begin{equation}\label{characterizeR}
{\cal R}_D=\{R+FA: dom\,A=im\,D,\, A\in{\cal L}(X)\},
\end{equation}
\begin{equation}\label{characterizeF}
{\cal F}_D=\{F(I-AD): dom\,A=im\,D,\, A\in{\cal L}(X)\},
\end{equation}
where    $R\in{\cal R}_{D}$ and $F\in{\cal F}_{D}$ are fixed
arbitrarily.

Let us illustrate the above concepts with two basic examples.
\begin{ex}\label{derivative} $X={\mathbb R}^{\mathbb R}$ - the linear space of all
functions, $D\in{\cal R}(X)$ - usual derivative, i.e.
 $Dx(t)\equiv x'(t)$, with $dom(D)\subset X$ consisting of all differentiable functions.
 Then, for an arbitrarily fixed $a\in\mathbb R$, by formula
 $Rx(t)=\int\limits_{a}^{t}x(s)ds$ one can define a right inverse $R\in{\cal
R}_{D}$ and the initial operator $F\in{\cal F}_D$ corresponding to
$R$ is given by $Fx(t)=x(a)$.
\end{ex}
\begin{ex}\label{difference}  $X={\mathbb {\mathbb
R}^{\mathbb N}}$ - the linear space of all sequences, $D\in{\cal
R}(X)$ - difference operator, i.e.
 $(Dx)_n=x_{n+1}-x_n$, for $n\in\mathbb N$. A right inverse $R\in{\cal
 R}_{D}$ is defined by the formulae $(Rx)_{1}=0$ and $(Rx)_{n+1}=\sum\limits_{i=1}^{n}x_{i}$  while
 $(Fx)_n=x_1$ defines the initial operator $F\in{\cal F}_D$ corresponding to $R$.
\end{ex}
An immediate consequence of Definition \ref{initial}, for an
invertible operator $D\in{\cal R}(X)$, i.e. $ker\, D=\{0\}$, is that
${\cal F}_{D}=\{0\}$. Therefore, the nontrivial initial operators do
exist only for operators which are right invertible but not
invertible. The family of all such operators is denoted by
\begin{equation}\label{nontriv}
{\cal R}^+(X)=\{D\in{\cal R}(X): dim\, Z(D) >0 \}.
\end{equation}
\begin{p}[Taylor Formula]
Suppose $D\in{\cal R}(X)$ and let $F\in{\cal F}_{D}$ be an initial
operator corresponding to $R\in{\cal R}_{D}$. Then the operator
identity
\begin{equation}\label{taylor}
I=\sum_{k=0}^{m}R^kFD^k + R^{m+1}D^{m+1}\, ,
\end{equation}
holds on $dom(D^{m+1})$,  for $m\in{\mathbb N}_0$.
\end{p}
\begin{proof}(Induction) See Ref.\cite{DPR}.
\end{proof}
Equivalent identity,   expressed as
\begin{equation}\label{xtaylor}
x=\sum_{k=0}^{m}R^kFD^kx + R^{m+1}D^{m+1}x\, ,
\end{equation}
for $x\in dom(D^{m+1})$ and $m\in{\mathbb N}_0$, is an algebraic
counterpart of the Taylor expansion formula,  commonly known in
mathematical analysis. The first component of the last formula
reflects the polynomial part while the second one can be viewed as
the corresponding reminder.

\begin{ex}To clearly demonstrate the resemblance of formula (\ref{xtaylor})
with the commonly used Taylor expression, we take $D$, $R$ and $F$
as in Example \ref{derivative}. Since there are many forms of the
reminders in use, it is more interesting to calculate the polynomial
part, which gives the well known result
$\sum\limits_{k=0}^{m}R^kFD^kx(t)=\sum\limits_{k=0}^{m}\frac{x^{(k)}(a)}{k!}(t-a)^{k}$,
for any function $x\in dom(D^{m+1})$.
\end{ex}
\begin{p}\label{nillpot}
Let $D\in {\cal R}(X)$ and $R\in{\cal R}_{D}$. Then $R$ is not a
nilpotent operator.
\end{p}
\begin{proof}
Suppose that $R^n\neq 0$ and $R^{n+1}=0$, for some $n\in\mathbb N$.
Then $0\neq R^n=IR^n=DRR^n=DR^{n+1}=0$, a contradiction.
\end{proof}
\begin{p}\label{nonemptyZm}
If $D\in{\cal R}^+(X)$, then $Z_m(D)\neq\emptyset$, for any
$m\in\mathbb N$.
\end{p}
\begin{proof}
The relation $Z_1(D)\neq\emptyset$ is straightforward. Let
$R\in{\cal R}_{D}$ and $z\in Z_1(D)$ be arbitrarily chosen elements.
Then, for any $m\in\mathbb N$, there is $R^{m-1}z\in Z_m(D)$.
\end{proof}
With right invertible operators possessing nontrivial kernels we
associate the following concept of $D$-polynomials.
\begin{df}\label{deg-polynomials}
If $D\in{\cal R}^+(X)$, then any element $u\in Z_{m+1}(D)$ is said
to be a $D$-polynomial of degree m, i.e. $deg\, u = m$, for
$m\in{\mathbb N}_0$. We assign no degree to the zero polynomial
$u\in Z_0(D)\equiv\{0\}$.
\end{df}
For the convenience' sake, one can also use the convention $deg\,
0=-\infty$.
\begin{p}\label{polsumRz}
If $D\in{\cal R}^+(X)$ and $R\in{\cal R}_{D}$, then for any
$D$-polynomial $u\in Z_{m+1}(D)$ there exist elements $z_0,\ldots
,z_m\in Z_1(D)$ such that
\begin{equation}\label{uRz}
u=z_0+Rz_1+\ldots +R^mz_m .
\end{equation}
\end{p}
\begin{proof}By formula (\ref{xtaylor}) we can write the
identity $u=\sum_{k=0}^{m}R^kFD^ku$ since $u\in Z_{m+1}(D)$ and
$R^{m+1}D^{m+1}u=0$. Then we define elements $z_k=FD^ku$,
$k=0,\ldots ,m\,$, which ends the proof.
\end{proof}

\begin{df}\label{monomials} Let $D\in{\cal R}^+(X)$ and $R\in{\cal R}_{D}$.
Then, any element $R^kz\in Z_{k+1}(D)$, for $z\in Z_1$, is said to
be an $R$-homogeneous $D$-polynomial (or $R$-monomial) of degree
$k\in{\mathbb N}_0$.
\end{df}
 Thus, any $D$-polynomial $u\in Z_{m+1}(D)$, of degree $deg\, u =
m$, is a sum of linearly independent $R$-homogeneous elements
$R^kz_k$, $k=0,\ldots ,m$. The linear space of all $D$-polynomials
is then
\begin{equation}\label{PolD}
P(D)=\bigcup\limits_{k=0}^{\infty}ker\, D^k
\end{equation}
whereas
\begin{equation}\label{PolDn}
P_n(D)=\bigcup\limits_{k=0}^{n}ker\, D^k=ker\, D^n
\end{equation}
 is the linear space of all $D$-polynomials of degree at most
$n\in{\mathbb N}_0$.

Let us fix a basis
\begin{equation}\label{zetabasis}
\{\zeta_s\}_{s\in S}\subset Z(D)\; ,
\end{equation}
  of the linear space $Z(D)$, $D\in{\cal R}^+(X)$, and define
\begin{equation}
Z^s(D)=Lin\{\zeta_s\}\, ,
\end{equation}
for $s\in S$. Then
\begin{equation}\label{ZDdirect}
Z(D)=\bigoplus_{s\in S}Z^s(D)\, .
\end{equation}
\begin{p}\label{polbasis}
For an arbitrary right inverse $R\in{\cal R}_{D}$, the family
$\{R^m\zeta_s\colon {s\in S}, m\in{\mathbb N}_0\}$ is the basis of
the linear space $P(D)$. Naturally,  $\{R^m\zeta_s\colon {s\in S},
m=0,1,\ldots n\}$ forms the basis of the linear space $P_n(D)$, for
$n\in{\mathbb N}_0$.
\end{p}
\begin{proof}Let $u=\sum_{i=1}^{k}\sum_{s\in S_{i}}a_{is} R^{m_i}\zeta_{s}$,
 $m_1<\ldots <m_k$ and  $S_i\subset S$ be finite subsets for
$i=1,\ldots ,k$. Assume $u=0$ and calculate $D^{m_k}u=\sum_{s\in
S_{k}}a_{ks}\zeta_{s}=0$, which implies $a_{ks}=0$, for all $s\in
S_k$. Hence $u=\sum_{i=1}^{k-1}\sum_{s\in S_{i}}a_{is}
R^{m_i}\zeta_{s}=0$ and analogously we get $D^{m_{k-1}}u=\sum_{s\in
S_{k-1}}a_{(k-1)s}\zeta_{s}=0$, which implies $a_{(k-1)s}=0$, for
all $s\in S_{k-1}$. Similarly we prove that $a_{is}=0$, for all
$s\in S_{i}$, $i=k-2,\ldots ,1$. Now, let $u\in P(D)$ be a
polynomial of degree $deg\,u=n \in{\mathbb N}_0\}$. Then, on the
strength of Proposition \ref{polsumRz}, we can write
$u=\sum_{k=0}^{n}R^kz_k $, for some elements $z_0,\ldots ,z_n\in
Z(D)$. In turn, each element $z_k$ can be expressed as a linear
combination $z_k=\sum_{s\in S_k}a_{ks}\zeta_s$, for some finite
subset of indices $S_k\subset S$. Hence we obtain
$u=\sum_{k=0}^{n}\sum\limits_{s\in S_k}a_{ks}R^k\zeta_s\,$.
\end{proof}

With a right inverse $R\in{\cal R}_{D}$, $s\in S$ and $n\in{\mathbb
N}_0$, we shall associate the linearly independent family
$\{R^m\zeta_s\colon m\in\{0,\ldots ,n\}\}$ forming a basis of the
linear space of $s$-homogeneous $D$-polynomials
\begin{equation}\label{Vsndirect}
V_{s}^{n}(D)=Lin\{R^m\zeta_s\colon m\in\{0,\ldots ,n\}\}\, ,
\end{equation}
(independent of the choice of $R$) of dimension
\begin{equation}\label{dimVsn}
dim\,V_{s}^{n}(D)=n+1\, ,
\end{equation}
being a linear subspace of $P_n(D)$. Then, on the strength of
Proposition \ref{polbasis}, the linear space $P_n(D)$ is a direct
sum
\begin{equation}\label{Pndirect}
P_n(D)=\bigoplus_{s\in S}V_s^n(D)\, .
\end{equation}
\begin{corol}If $dim\, Z(D)<\infty$, the following formula holds
\begin{equation}\label{dimpoln}
dim\, P_n(D)=(n+1)\cdot dim\, Z(D)\, ,
\end{equation}
for any $n\in{\mathbb N}_0$.
\end{corol}
Naturally, one can extend formula (\ref{Vsndirect}) and define
\begin{equation}\label{Vdirect}
V_{s}(D)=Lin\{R^m\zeta_s\colon m\in{\mathbb N}_0\}\, ,
\end{equation}
which is both $D$- and $R$-invariant subspace of $P(D)$, i.e.
\begin{equation}\label{DinvariantVsn}
DV_{s}(D)\equiv\{Du\colon u\in V_{s}(R)\}\subset V_{s}(D)\, ,
\end{equation}
\begin{equation}\label{RinvariantVsn}
RV_{s}(D)\equiv\{Ru\colon u\in V_{s}(D)\}\subset V_{s}(D)\, .
\end{equation}
Thus,  $P(D)$ turns out to be simultaneously $D$- and $R$-invariant
linear subspace of $X$, since it can be decomposed as the following
direct sum
\begin{equation}\label{PDVsdirect}
P(D)=\bigoplus_{s\in S}V_s(D)\, .
\end{equation}
 Since $P(D)$ is a linear subspace of $X$, there exists
(not uniquely) another linear subspace $Q(D)$ of $X$ such that
\begin{equation}\label{PsumQ}
X=P(D)\oplus Q(D)\, .
\end{equation}
Then, every linear mapping $\phi\colon X\rightarrow X$ can be
decomposed as the direct sum
\begin{equation}\label{phiPQ}
\phi=\phi_{_P}\oplus\phi_{_Q}\, ,
\end{equation}
of two restrictions $\phi_{_P}=\phi_{|P(D)}$ and
$\phi_{_Q}=\phi_{|Q(D)}$, i.e. for any $x'\in P(D)$ and $x''\in
Q(D)$ there is $\phi (x'+x'')=\phi_{_P}(x')+\phi_{_Q}(x'')$. In
particular, the mappings  $D\in{\cal R}^+(X)$, $R\in{\cal R}_{D}$
can be decomposed as direct sums $D=D_{_P}\oplus D_{_Q}$,
$R=R_{_P}\oplus R_{_Q}$ such that
\begin{equation}
I=DR=D_{_P}R_{_P}\oplus D_{_Q}R_{_Q}=I_{_P}\oplus I_{_Q}\, ,
\end{equation}
\begin{equation}
RD=R_{_P}D_{_P}\oplus R_{_Q}D_{_Q}\, ,
\end{equation}
which allows for the decomposition of the initial operator $F$
corresponding to $R$
$$
F=I-RD=I_{_P}\oplus I_{_Q}-R_{_P}D_{_P}\oplus R_{_Q}D_{_Q}=
$$
\begin{equation}
=(I_{_P}-R_{_P}D_{_P})\oplus (I_{_Q}-R_{_Q}D_{_Q})=F_{_P}\oplus
F_{_Q}.
\end{equation}
\begin{p}
Let $D\in{\cal R}^+(X)$, $R',R''\in{\cal R}_{D}$ be any right
inverses and $F',F''\in{\cal F}_{D}$ be the initial operators
corresponding to $R'$ and $R''$, respectively. Then
$R:=R'_{_P}\oplus R''_{_Q}\in{\cal R}_{D}$ and $F:=F'_{_P}\oplus
F''_{_Q}\in{\cal F}_{D}$ corresponds to $R$.
\end{p}
\begin{proof}
$$DR=(D_{_P}\oplus D_{_Q})(R'_{_P}\oplus
R''_{_Q})=D_{_P}R'_{_P}\oplus D_{_Q}R''_{_Q}=I_{_P}\oplus I_{_Q}=I\,
,$$
$$RD=(R'_{_P}\oplus R''_{_Q})(D_{_P}\oplus D_{_Q})=R'_{_P}D_{_P}\oplus R''_{_Q}D_{_Q}\,
,$$

$$F=F'_{_P}\oplus F''_{_Q}=(I_{_P}-R'_{_P}D_{_P})\oplus
(I_{_Q}-R''_{_Q}D_{_Q})=
$$
$$=
I_{_P}\oplus I_{_Q}-R'_{_P}D_{_P}\oplus R''_{_Q}D_{_Q}=I-RD\, .
$$
\end{proof}
The last results allow one to combine right inverses and initial
operators as direct sums of independent components.
\section{Divided difference operators in $(\sigma ,\tau )$-quantum
calculus}\label{Divided}
Quantum calculus  is based on a difference operator, called also a
quantum differential, defined as
\begin{equation}\label{operator-d}
d_{h'q'}^{hq}f(x)=f(qx+h)-f(q'x+h'),
\end{equation}
with the natural assumption that either $q\neq q'$ or $h\neq h'$.
Let us denote by $e$ the identity function, i.e. $e(x)\equiv x$, and
define a divided difference operator $D_{h'q'}^{hq}$ by formula
\begin{equation}\label{operator-D}
D_{h'q'}^{hq}f(x)=\frac{df(x)}{de(x)}\equiv\frac{f(qx+h)-f(q'x+h')}{(q-q')x+h-h'}\,
.
\end{equation}
 We shall refer to $D_{h'q'}^{hq}$ as the quantum
derivative operator \cite{kac}. If the parameters are known from
context, the simplified notation $d\equiv d_{h'q'}^{hq}$ and
$D\equiv D_{h'q'}^{hq}$ is used.

The following four cases of quantum calculus are the most common
ones:
\begin{enumerate}
\item $h$-calculus when $h\neq 0$, $h'=0$ and $q=q'=1$  ,
\item $q$-calculus when $q\neq 1$,  $q'=1$ and $h=h'=0$,
\item $h$-symmetric calculus when $h=-h'\neq 0$ and
$q=q'=1$,
\item $q$-symmetric calculus when $q=q'^{-1}\neq 1$ and
$h=h'=0$.
\end{enumerate}
The expressions $qx+h$, $q'x+h'$ can be replaced by more general
ones $\tau (x)$, $\sigma (x)$ and formulae (\ref{operator-d}),
(\ref{operator-D}) can be rewritten as
\begin{equation}\label{operator-dtau}
d_{\sigma}^{\tau}f(x)=f(\tau(x))-f(\sigma(x))\, ,
\end{equation}
and
\begin{equation}\label{operator-Dtau}
D_{\sigma}^{\tau}f(x)=\frac{d_{\sigma}^{\tau}f(x)}{d_{\sigma}^{\tau}e(x)}\equiv
\frac{f(\tau(x))-f(\sigma(x)}{\tau(x)-\sigma(x)}\, .
\end{equation}
If the mappings $\sigma$ and $\tau$ are known from context, we shall
use the simplified notation $d\equiv d_{\sigma}^{\tau}$ and $D\equiv
D_{\sigma}^{\tau}$.

To prevent the denominator of formula (\ref{operator-Dtau}) from
being zero, all functions considered above are defined on the
naturally restricted domain \begin{equation}\label{domainM}
M=\{x\in{\mathbb R}:\sigma(x)\neq\tau (x) \}\, .
\end{equation}
The fixed mappings $\sigma ,\tau :M\rightarrow M$, together with the
domain $M\subseteq\mathbb R$, specify the type of quantum calculus
considered.

In this paper we study a generalization of the  quantum calculus
presented in Ref. \cite{kac}. We assume $M$ to be an arbitrary set
and fix two mappings $\sigma ,\tau  :M\rightarrow M$ corresponding
to those mentioned above.

However, $\sigma$ and $\tau$ are not real-valued maps, in general.
Therefore, to adapt formula (\ref{operator-Dtau}) we need a numeric
expression that will play the role of a corresponding denominator.
For that purpose we endow $M$ with a tension structure, defined with
the help of one or more tension functions \cite{MULT1}.
\begin{df}\label{theta} A function
$\theta :M\times M\rightarrow\mathbb R$ is said to be a tension
function if
\begin{equation}
\theta (p_1,p_2)+\theta (p_2,p_3)=\theta (p_1,p_3) ,
\end{equation}
for any $p_1, p_2, p_3\in M$.
\end{df}
Directly from definition it follows that a linear combination of
tension functions defined on $M\times M$ is a tension function
again. Thus, any family of tension functions defined on $M\times M$
generates the corresponding linear space.
\begin{df}\label{structureT}
Any linear space $T$ of tension functions defined on $M\times M$ is
said to be a tension structure on $M$ and the pair $(M,T)$ is called
the tension space of dimension $dim\, T$.
\end{df}
Directly from the above definition, one can show that any tension
function $\theta\in T$ is skew symmetric, i.e.
\begin{equation}
\theta (p_1,p_2) = - \,\theta (p_2,p_1)\, ,
\end{equation}
for any $p_1, p_2\in M$. With any $\theta\in T$ we associate
functions $\theta_q:M\rightarrow\mathbb R$ defined by
\begin{equation}\label{theta_q}
\theta_q (p)=\theta (p,q),
\end{equation}
for any $p, q\in M$. Intuitively,  $\theta_q$ plays the role of a
potential function defined on $M$, associating with any point $p\in
M$ the scalar potential $\theta_q(p)$ such that $\theta_q(q)=0$.

Let $(M,T)$ be a one-dimensional tension space with a tension
structure $T$ generated by a single tension function $\theta$ and
assume that $\theta (\tau (p),\sigma (p))\neq 0$, for any $p\in M$.
Then we define the $(\sigma ,\tau )$-quantum difference operator
$d\equiv d_{\sigma}^{\tau}$
\begin{equation}\label{dsigmatau}
df(p)=f(\tau (p))-f(\sigma (p)),
\end{equation}
and the $(\sigma ,\tau )$-quantum derivative operator $D\equiv
D_{\sigma}^{\tau}$

\begin{equation}\label{Dsigmatau}
Df(p)=\frac{df(p)}{d\theta_q(p)}\equiv
\frac{f(\tau(p))-f(\sigma(p))}{\theta(\tau(p),\sigma(p))}\, ,
\end{equation}
for any $p\in M$ (independently of $q\in M$). The following Leibniz
rule can be checked easily
\begin{equation}\label{leibniz}
D(f\cdot g)(p)=f(\tau (p))\cdot D(g)(p) + D(f)(p)\cdot g(\sigma
(p))\, ,
\end{equation}
for  any $f, g: M\rightarrow\mathbb R$.
\section{$D$-polynomials}\label{D-polynomials}
\begin{df}\label{thetadirected}
A mapping $\rho : M\rightarrow M$ is said to be rightward
$\theta$-directed if
\begin{equation}\label{rightward}
\theta(p,\rho (p))<0\, ,
\end{equation}
 and it is said to be leftward $\theta$-directed if
\begin{equation}\label{leftward}
\theta(p,\rho (p))>0\, ,
\end{equation}
 for any $p\in M$. We say that $\rho$ is a
 $\theta$-directed mapping if one of the above conditions holds.
\end{df}
Assume the notation: $\rho^0=id_M$ and $\rho^n=\rho\circ\rho^{n-1}$,
for any $n\in\mathbb N\,$.
\begin{p}
For any $\theta$-directed mapping $\rho : M\rightarrow M$,
$n\in\mathbb N$, the composition $\rho^n$ has no fixed points, i.e.
\begin{equation}\label{nofixed}
\rho^n(p)\neq p\, ,
\end{equation}
for $p\in M\,$.
\end{p}
\begin{proof} Let $\rho $ be a rightward $\theta$-directed mapping. Then we
have inequalities $\theta (\rho (p), p)>0$, ... , $\theta
(\rho^n(p),\rho^{n-1}(p))>0$, for any $n\in\mathbb N$ and $p\in M$.
Consequently,
$$
\theta (\rho^n(p),p)=\theta (\rho^n(p),\rho^{n-1}(p))+\ldots +\theta
(\rho (p), p)>0\, .
$$
Analogously, for a leftward $\theta$-directed mapping we show that
$\theta (\rho^n(p),p)<0\,$, for any $n\in\mathbb N$ and $p\in
M$.\hfill
\end{proof}
 Let us notice that condition (\ref{nofixed}) is not
a consequence of the weaker assumption that $\theta (\rho (p),p)\neq
0$, for any $p\in M$. In that case there would be $\rho (p)\neq p$
but not necessarily $\rho^n (p)\neq p\,$, for any $n\in\mathbb N$
and $p\in M$.
\begin{df}
We say that $\theta$ is homogeneous with respect to $\rho$ (shortly,
$\rho$-homogeneous) if there exists $r\in\mathbb R$, the so-called
$\rho$-homogeneity coefficient, such that
\begin{equation}\label{homogeneous}
\theta (\rho (p_1),\rho (p_2))=r\cdot\theta (p_1,p_2),
\end{equation}
for any $p_1, p_2\in M$.
\end{df}
\begin{p}\label{positivecoeff}Let $\rho :M\rightarrow M$ be a $\theta$-directed mapping
and $\theta$ be a $\rho$-homogeneous tension function. Then, for the
$\rho$-homogeneity coefficient we get
$r>0\, $.
\end{p}
\begin{proof} Directly from Definition (\ref{thetadirected}) we get
$r\neq 0$. Since $r$ is a $\rho$-homogeneity coefficient for
$\theta$ and $\rho$ is a $\theta$-directed mapping, $\theta
(\rho^2(p),\rho (p))$ and $\theta (\rho (p), p)$ are of common sign
and $\theta (\rho^2(p),\rho (p))=r\cdot\theta (\rho (p), p)$. Hence
we conclude that $r>0$.
\end{proof}

Let $c\in\mathbb R$ and define the constant function
$\hat{c}:M\rightarrow\mathbb R$, $\hat{c}(p):=c\,$. Evidently,
$\hat{c}\in ker\, D$, for any $c\in\mathbb R$, which means that
$dim\, ker\,D >0$. Hence, by formula (\ref{nontriv}) we can write
$D\in{\cal R}^+(X)$, for $X={\mathbb R}^M$.

Now, by using Definition \ref{deg-polynomials}, let us analyze the
elements of $Z_{m+1}(D)$, i.e. $D$-polynomials of degree
$m\in\mathbb N\cup\{0\}$. Since $D(\hat{1})=0$, the unitary constant
function $\hat{1}:M\rightarrow\mathbb R$ is a $D$-polynomial of
degree $deg\,\hat{1}=0$, i.e. $\hat{1}\in Z_1(D)$.

On the strength of Proposition \ref{polsumRz}, the explicit form of
$D$-polynomials can be obtained from formula (\ref{uRz}), provided a
right inverse $R\in\cal R_{D}$ is known explicitely. If $\sigma$ and
$\tau$ are two commuting bijections, a concrete example of a right
inverse of $D$ is shown in \cite{MULT1}.

 Then, let us fix a point $q\in M$ and define the
following sequence of functions
$\theta_{q}^{(m)}:M\rightarrow\mathbb R$, $m\in\mathbb N_0$,
\begin{equation}\label{Wkq}
\theta_{q}^{(0)}=\hat{1}\;\;\; \mbox{and}\;\;\;
\theta_{q}^{(m)}=\prod_{k=1}^{m}\theta_{\tau^{m-k}\sigma^{k-1}(q)}\,
,
\end{equation}
e.g. for $p\in M$,  $\theta_{q}^{(1)}(p)=\theta (p,q)\,$,
$\theta_{q}^{(2)}(p)=\theta (p,\tau (q))\theta (p,\sigma
(q))\,$,\linebreak $\theta_{q}^{(3)}(p)=\theta (p,\tau^2 (q))\theta
(p,\tau\sigma (q))\theta (p,\sigma^2 (q))\,$, etc.
\begin{p}\label{n}
Let $\sigma ,\tau :M\rightarrow M$ be commuting maps and $\theta$ be
a tension function homogeneous with respect to $\sigma$ and $\tau$
with the homogeneity coefficients $s$ and $t$, respectively, such
that $\theta (\tau (p),\sigma (p))\neq 0$ for any $p\in M$. Then,
for any $n\in\mathbb N$, the following formula is true
\begin{equation}\label{DWn}
D\theta_{q}^{(n)}=[n]_\sigma^\tau\cdot \theta_{q}^{(n-1)}\, ,
\end{equation}
where
\begin{equation}\label{[n]}
 [n]_\sigma^\tau = \sum\limits_{k=1}^{n}t^{n-k}s^{k-1}\, .
 \end{equation}
\end{p}
\begin{proof}
$$d \theta_{q}^{(n)}(p)=
\theta_{q}^{(n)}(\tau (p))-\theta_{q}^{(n)}(\sigma(p))=
$$
$$=
\prod\limits_{k=1}^{n}\theta (\tau (p),\tau^{n-k}\sigma^{k-1}(q))
-\prod\limits_{k=1}^{n}\theta (\sigma (p),
\tau^{n-k}\sigma^{k-1}(q))=
$$
$$=
\theta(\tau(p),\sigma^{n-1}(q)) \prod\limits_{k=1}^{n-1}\theta (\tau
(p),\tau^{n-k}\sigma^{k-1}(q))-
$$
$$
-\theta(\sigma(p),\tau^{n-1}(q)) \prod\limits_{k=2}^{n}\theta
(\sigma (p), \tau^{n-k}\sigma^{k-1}(q))=
$$
$$=
\left(t^{n-1}\theta(\tau(p),\sigma^{n-1}(q))-s^{n-1}\theta(\sigma(p),\tau^{n-1}(q))\right)
\prod\limits_{k=1}^{n-1}\theta(p,\tau^{n-1-k}\sigma^{k-1}(q))
$$
$$=
\left(t^{n-1}\theta(\tau(p),\sigma^{n-1}(q))-s^{n-1}\theta(\sigma(p),\tau^{n-1}(q))\right)
\theta_{q}^{(n-1)}(p)\, .
$$
In turn,
$$
t^{n-1}\theta(\tau(p),\sigma^{n-1}(q))=
t^{n-1}\theta(\tau(p),\sigma(p))
+t^{n-1}\theta(\sigma(p),\sigma^{n-1}(q))=
$$
$$
=t^{n-1}\theta(\tau(p),\sigma(p))
+t^{n-2}s\theta(\tau(p),\sigma^{n-2}\tau(q))=
$$
$$
=(t^{n-1}+t^{n-2}s)\theta(\tau(p),\sigma(p))+
t^{n-3}s^{2}\theta(\tau(p),\sigma^{n-3}\tau^{2}(q))=\ldots
$$
$$
\ldots =(t^{n-1}+\ldots +t^{n-i}s^{i-1})\theta(\tau(p),\sigma(p))+
t^{n-i}s^{i-1}\theta(\tau(p),\sigma^{n-i}\tau^{i-1}(q))=\ldots
$$
$$
\ldots
=\sum\limits_{i=1}^{n}t^{n-i}s^{i-1}\theta(\tau(p),\sigma(p))+
s^{n-1}\theta(\tau(p),\tau^{n-1}(q))\, .
$$
Finally,
$$
D\theta_{q}^{(n)}(p)= \frac{d\theta_{q}^{(n)}(p)}{\theta
(\tau(p),\sigma(p))}=\sum\limits_{i=1}^{n}t^{n-i}s^{i-1}\cdot
\theta_{q}^{(n-1)}(p)=[n]_{\sigma}^{\tau}\cdot
\theta_{q}^{(n-1)}(p)\, .
$$
\end{proof}
Since there is $D\theta_{q}^{(0)}=\hat{0}$ (zero constant function),
formula (\ref{DWn}) can be extended to the case $n=0$, if we define
$[0]_\sigma^\tau=0$. Since $\sigma$ and $\tau$ are fixed mappings
giving rise to a concrete type of a quantum calculus considered, the
indices $\sigma$ and $\tau$ will be omitted, i.e. the notation
$[n]\equiv [n]_{\sigma}^{\tau}$ will be used in the sequel. With the
symbol $[n]$ we associate the $(\sigma ,\tau )$-quantum factorial
defined as
\begin{equation}\label{factorial}
[n]!=
\begin{cases}1& if\;\; n=0\\
[n]\cdot [n-1]!& if\;\; n=1,2,\ldots\, .
\end{cases}
\end{equation}
 An immediate consequence of the last proposition is that each
function $\theta_{q}^{(n)}\colon M\rightarrow\mathbb R$ is a
representative element of $Z_{n+1}(D)$, i.e. it is a $D$-polynomial
of degree $deg\, \theta_{q}^{(n)}=n$, for $n\in{\mathbb N}_0$.
\begin{p}For any function $\zeta\in Z(D)$ and any mapping $\chi\colon M\rightarrow M$
commuting with $\sigma$ and $\tau$, there is also $\zeta\circ\chi\in
Z(D)$.
\end{p}
\begin{proof}
Since $\zeta\in Z(D)$, there is $\zeta\circ\tau = \zeta\circ\sigma$
and therefore $\zeta\circ\tau\circ\chi = \zeta\circ\sigma\circ\chi$.
The commutativity of $\sigma$, $\tau$ with $\chi$ implies that
$\zeta\circ\chi\circ\tau = \zeta\circ\chi\circ\sigma$ or
equivalently $\zeta\circ\chi\in Z(D)$.
\end{proof}
With any basis $\{\zeta_s\colon s\in S\}$ of the linear space $Z(D)$ one can always associate a family
$\{q_s\in M\colon s\in S\}$ of points such that
$\zeta_{s}(q_{s})\neq 0$ and $\zeta_{s_1}(q_{s_2})=0$ whenever $s_1\neq s_2$, for $s, s_1, s_2\in S$.
Then, one can always assume that functions $\zeta_s$ are normalized in such a way that
\begin{equation}\label{zetanormaliz}
\zeta_{s}(q_{t})=\begin{cases}
1& if\;\; s=t\\
0& if\;\; s\neq t\, ,
\end{cases}
\end{equation}
for any $s, t\in S$.

Let us define functions
\begin{equation}\label{zeta-k-s-def}
\zeta_{q_{s}s}^{(k)}=\frac{1}{[k]!}(\zeta_s\circ\sigma^{-k})\cdot\theta_{q_{s}}^{(k)}\,
,
\end{equation}
for ${s\in S}$, $k\in{\mathbb N}_0$. Then by formula (\ref{leibniz})
and Proposition \ref{n} we obtain
\begin{equation}
D\zeta_{q_{s}s}^{(k)}=\frac{1}{[k]!}(\zeta_s\circ\sigma^{-k}\circ\sigma)\cdot
D\theta_{q_{s}}^{(k)}=\zeta_{q_{s}s}^{{(k-1)}}\, .
\end{equation}
The consequence of the last formula is that $\zeta_{q_{s}s}^{(k)}\in
Z_{k+1}(D)$, i.e. $\zeta_{q_{s}s}^{(k)}$ is a $D$-polynomial of
degree $deg\,\zeta_{q_{s}s}^{(k)}=k$, for any $s\in S$ and
$k\in{\mathbb N}_0$.
\begin{p}
The family $\{\zeta_{q_{s}s}^{(k)}\colon s\in S, k\in{\mathbb
N}_0\}$ is a basis of the linear space $P(D)$ and
$\{\zeta_{q_{s}s}^{(k)}\colon s\in S,\; k=0,\ldots ,n\}$ is a basis
of the linear space $P_n(D)$.
\end{p}
\begin{proof} Assume that $\sum\limits_{i=0}^{k}\sum\limits_{s\in S_i}
a_{is}\zeta_{q_{s}s}^{(i)}=\hat{0}$, for some coefficients $a_{is}$,
 $S_i\subset S$ - finite subsets, $i=0,\ldots ,k$. Then we
 calculate
 $D^{k}\sum\limits_{i=0}^{k}\sum\limits_{s\in S_i}
a_{is}\zeta_{q_{s}s}^{(i)}=\sum\limits_{s\in S_k}
a_{ks}\zeta_{s}=\hat{0}$ which implies that $a_{ks}=0$ for $s\in
S_k$. Hence we can write $\sum\limits_{i=0}^{k-1}\sum\limits_{s\in
S_i} a_{is}\zeta_{q_{s}s}^{(i)}=\hat{0}$. Similarly, we  calculate
$D^{k-1}\sum\limits_{i=0}^{k-1}\sum\limits_{s\in S_i}
a_{is}\zeta_{q_{s}s}^{(i)}=\sum\limits_{s\in S_{k-1}}
a_{(k-1)s}\zeta_{s}=\hat{0}$ and show that $a_{(k-1)s}=0$ for $s\in
S_{k-1}$. Analogously, step by step we prove that $a_{ms}=0$ for
$s\in S_m$ and $m=k-2,...,1$. Now, let $u\in P(D)$ be of degree
$deg\, u=k$. We will show that
$u=\sum\limits_{i=0}^{k}\sum\limits_{s\in S_i}
a_{is}\zeta_{q_{s}s}^{(i)}$. Let us calculate:
$$D^ku=\sum\limits_{i=0}^{k}\sum\limits_{s\in S_i}
a_{is}D^k\zeta_{q_{s}s}^{(i)}=\sum\limits_{s\in S_k}
a_{ks}\zeta_{s}.$$ Then, using formula (\ref{zetanormaliz}) we
obtain
\begin{equation}
 a_{ks}= D^ku(q_s)\, ,
 \end{equation}
for $s\in S_k$. Other coefficients $a_{ms}$, for $0\leq m <k$, we
can calculate from the recursive formula
\begin{equation}\label{ams-recursion}
a_{ms}=D^mu(q_s)-\sum_{j=1}^{k-m}\sum_{t\in
S_{m+j}}a_{(m+j)t}\zeta_{q_tt}^{(j)}(q_s)\, .
\end{equation}
\end{proof}
One can select (not uniquely) a linear subspace $Q(D)$ in
$X={\mathbb R}^M$ such that
\begin{equation}\label{XDsumQ}
X=P(D)\oplus Q(D)\, .
\end{equation}
According to formula (\ref{phiPQ}), any right inverse $R\in{\cal
R}_D$ can be decomposed as a direct sum
\begin{equation}\label{Rdecomposed}
R=R_{P}\oplus R_{Q}
\end{equation}
of its restrictions $R_{P}=R_{|P(D)}$ and $R_{Q}=R_{|Q(D)}$. The
component $R_P$ can be defined on the basis
$\{\zeta_{q_ss}^{(k)}\colon s\in S, k\in{\mathbb N}_0\}$ by formula
\begin{equation}\label{R_Pdefinition}
R_{P}\zeta_{q_ss}^{(k)}=\zeta_{q_ss}^{(k+1)}\, ,
\end{equation}
for any ${s\in S}$, $k\in{\mathbb N}_0$. Concerning the component
$R_Q$, we can assume any definition.
 Then, on the strength of formula (\ref{Vdirect}),
the linear space $P(D)$ is a direct sum of $D$- and $R$-invariant
linear subspaces $V_s(R)$, $s\in S$. The initial operator $F$
corresponding to the above $R$ is also a direct sum
\begin{equation}\label{Fdecomposed}
F=F_P\oplus F_Q\, ,
\end{equation}
with the components $F_P=F_{|P(D)}$ and $F_Q=F_{|Q(D)}$ given by
formula (\ref{FbyR}). For $F_{P}$ we obtain the explicit formula
\begin{equation}
F_{P}\zeta_{q_ss}^{(k)}=(I-R_{P}D_\sigma^\tau)\zeta_{q_ss}^{(k)}=\begin{cases}\zeta_s& if\;\; k=0\\
\hat{0}& if\;\; k=1,2,\ldots\, ,
\end{cases}
\end{equation}
where $s\in S$.
\section{Taylor formula in $(\tau ,\sigma )$-quantum calculus}\label{Taylor}
Let $R$ be a right inverse of $D=D_\sigma^\tau$ and $F$ be the
initial operator corresponding to $R$. Then, according to formula
(\ref{taylor}), we have the Taylor formula
\begin{equation}
I=\sum_{j=0}^{n}R^jF{D}^j + R^{n+1}{D}^{n+1}\, ,
\end{equation}
which holds on $dom{(D)}^{n+1}$,  for $n\in\mathbb N_0$. Now,
suppose that $W\in Z_{n+1}(D)$, i.e. $W\in P(D)$ is a $D$-polynomial
of degree $deg W=n$. Then $(D)^{n+1}W=\hat{0}$ and
\begin{equation}
W=\sum_{k=0}^{n}R^kF{D}^kW \, .
\end{equation}
Since $F{(D)}^kW\in\, Z(D)$, there exists a finite subset
$S_W^k\subset S$ and coefficients $\lambda_{ks}\in\mathbb R$, $s\in
S_W^k$, such that
\begin{equation}
F{(D)}^kW=\sum_{s\in S_W^k} \lambda_{ks}\zeta_s\, .
\end{equation}
 Thus, we obtain the formula
\begin{equation}\label{W-with-R-zeta-s}
W=\sum_{k=0}^{n}\sum_{s\in S_W^k}\lambda_{ks} R^k\zeta_s \, .
\end{equation}
The coefficients $\lambda_{ks}$ are given by
\begin{equation}
\lambda_{ks}=(F{D}^kW)(q_s)\, ,
\end{equation}
which allows us to write
\begin{equation}
W=\sum_{k=0}^{n}\sum_{s\in S_W^k}(F{D}^kW)(q_s) R^k\zeta_s \, .
\end{equation}
Define functions $\Lambda_{q_ss}^{(m)}\colon M\rightarrow\mathbb R$
recursively as $\Lambda_{q_ss}^{(0)}=\zeta_s$ and
\begin{equation}\label{Lambda-qss-def}
\Lambda_{q_ss}^{(m)}=\zeta_{q_ss}^{(m)}-\sum_{i=0}^{m-1}\zeta_{q_ss}^{(m-i)}(q_s)
\Lambda_{q_ss}^{(i)}\, ,
\end{equation}
for $q_s\in M$, $s\in S$ and $m\in\mathbb N$. One can easily check
that
\begin{equation}
\Lambda_{q_ss}^{(m)}(q_s)=\hat{0}\, ,
\end{equation}
\begin{equation}
D\Lambda_{q_ss}^{(m)}=\Lambda_{q_ss}^{(m-1)} \, ,
\end{equation}
for any $s\in S$, $m\in{\mathbb N}$. Hence, for any $s\in S$, the
family $ \{\Lambda_{q_ss}^{(m)}\colon m\in{\mathbb N}_0\} $ is
linearly independent and forms a basis of the linear space of
$s$-homogeneous $D$-polynomials
\begin{equation}\label{VbyRLambda}
V_s(D)=Lin\{\Lambda_{q_ss}^{(m)}\colon m\in{\mathbb N}_0\}\, .
\end{equation}
On the strength of formula (\ref{PDVsdirect}) there is
$P(D)=\bigoplus\limits_{s\in S}V_s(D)$ and there exists another
subspace $Q(D)$ such that $X=P(D)\oplus Q(D)$. Below we shall use
the projection mappings
\begin{equation}\label{s-proiection}
\pi_s\colon P(D)\rightarrow V_s(D)\, ,
\end{equation}
for $s\in S$.

Let us take a right inverse $R\in {\cal R}_D$ defined on $P(D)$ by
\begin{equation}\label{LambdaR}
R\Lambda_{q_ss}^{(m)}=\Lambda_{q_ss}^{(m+1)}\, ,
\end{equation}
for any $s\in S$, $m\in{\mathbb N}_0$, while its definition on
$Q(D)$ can be any. Then we can write
\begin{equation}
\Lambda_{q_ss}^{(m)}=R^m\zeta_s\, ,
\end{equation}
for any $s\in S$, $m\in{\mathbb N}_0$. Therefore formula
(\ref{W-with-R-zeta-s}) can be written as
\begin{equation}\label{W-with-Lambda-s}
W=\sum_{k=0}^{n}\sum_{s\in S_W^k}\lambda_{ks} \Lambda_{q_ss}^{(k)}
\, .
\end{equation}
By using projections (\ref{s-proiection}),  we obtain the components
$W_s$ of $W$ defined as
\begin{equation}\label{Ws-component}
W_s\equiv\pi_s (W)=\sum_{k=0}^{n}\lambda_{ks} \Lambda_{q_ss}^{(k)},
\end{equation}
for any $s\in S_W\equiv\bigcup\limits_{k=0}^{n}S_W^k$. Naturally,
there is
\begin{equation}\label{WbyWs}
W=\sum_{s\in S_W}W_s\, .
\end{equation}
The coefficients $\lambda_{ks}$ in formula (\ref{Ws-component}) can
be computed as
\begin{equation}\label{lambda-ks-computed}
\lambda_{ks}=D^kW_s(q_s)\, ,
\end{equation}
for any $k\in{\mathbb N}_0$ and $s\in S_W$. Finally we obtain the
Taylor formulae
\begin{equation}\label{TaylorWs}
W_s=\sum_{k=0}^{n}D^kW_s(q_s) \Lambda_{q_ss}^{(k)}\, ,
\end{equation}
for $s\in S_W$, and
\begin{equation}\label{TaylorW}
W=\sum_{s\in S_W}\sum_{k=0}^{n}D^kW_s(q_s) \Lambda_{q_ss}^{(k)}\, .
\end{equation}
Naturally, if $W\in V_s(D)$, for some $s\in S$, there is
\begin{equation}\label{TaylorWinVs}
W=\sum_{k=0}^{n}D^kW(q_s) \Lambda_{q_ss}^{(k)}\, .
\end{equation}
In the particular case, for q-calculus or its symmetric version the
corresponding results one can find in Ref.\cite{kac}.

\end{document}